\documentclass[english]{article}
\usepackage[bmargin=3.8cm]{geometry}
\usepackage{hyperref}
\usepackage{amsthm,amsfonts,amssymb,amsmath}
\usepackage{graphicx}

\newtheorem{theorem}{Theorem}[section]
\newtheorem{proposition}[theorem]{Proposition}
\newtheorem{lemma}[theorem]{Lemma}

\newtheorem{corollary}[theorem]{Corollary}

\theoremstyle{definition}
\newtheorem{definition}[theorem]{Definition}

\numberwithin{figure}{section}
\allowdisplaybreaks

\usepackage{babel}
\usepackage[T1]{fontenc}
\usepackage[utf8]{inputenc}

\newcommand{\su}{\subseteq}

\newcommand{\F}{\mathbb{F}}

\newcommand{\Fpn}{\mathbb{F}_p^{n}}
\newcommand{\Z}{\mathbb{Z}}
\newcommand{\s}{\mathfrak{s}}
\newcommand{\M}{\mathcal{M}}

\begin{document}
\title{On the size of subsets of $\Fpn$ without $p$ distinct elements summing to zero}
\author{Lisa Sauermann\thanks{Department of Mathematics, Stanford University, Stanford, CA 94305. Email: {\tt lsauerma@stanford.edu}.}}

\maketitle

\begin{abstract}
Let us fix a prime $p$. The Erd\H{o}s-Ginzburg-Ziv problem asks for the minimum integer $s$ such that any collection of $s$ points in the lattice $\mathbb{Z}^n$ contains $p$ points whose centroid is also a lattice point in $\mathbb{Z}^n$. For large $n$, this is essentially equivalent to asking for the maximum size of a subset of $\Fpn$ without $p$ distinct elements summing to zero.

In this paper, we give a new upper bound for this problem for any fixed prime $p\geq 5$ and large $n$. In particular, we prove that any subset of $\Fpn$ without $p$ distinct elements summing to zero has size at most $C_p\cdot \left(2\sqrt{p}\right)^n$, where $C_p$ is a constant only depending on $p$. For $p$ and $n$ going to infinity, our bound is of the form $p^{(1/2)\cdot (1+o(1))n}$, whereas all previously known upper bounds were of the form $p^{(1-o(1))n}$ (with $p^n$ being a trivial bound).

Our proof uses the so-called multi-colored sum-free theorem which is a consequence of the Croot-Lev-Pach polynomial method. This method and its consequences were already applied by Naslund as well as by Fox and the author to prove bounds for the problem studied in this paper. However, using some key new ideas, we significantly improve their bounds.
\end{abstract}

\section{Introduction}

For given positive integers $m$ and $n$, what is the minimum integer $s$ such that among any $s$ points in the integer lattice $\mathbb{Z}^n$ there are $m$ points whose centroid is also a lattice point in $\mathbb{Z}^n$? This question has been raised by Harborth \cite{harborth} in 1973 and the resulting minimum integer $s$ can be interpreted as the Erd\H{o}s-Ginzburg-Ziv constant $\s(\Z_m^n)$ of $\Z_m^n$. The study of Erd\H{o}s-Ginzburg-Ziv constants was initiated by a result of Erd\H{o}s, Ginzburg and Ziv \cite{egz} from 1961 stating that among any set of $2m-1$ integers there are $m$ elements whose average is also an integer. This implies that $\s(\Z_m^1)=2m-1$ and thus solves the question above for dimension $n=1$. For $n=2$, Reiher \cite{reiher} established that $\s(\Z_m^2)=4m-3$. More generally, for any fixed dimension $n$, Alon and Dubiner \cite{alondubiner} proved that $\s(\Z_m^n)$ grows linearly with $m$. On the other hand, bounding $\s(\Z_m^n)$ for fixed $m$ and large $n$ remains a wide open problem which has received a lot of attention (see for example \cite{alondubiner, edel2, edeletal, elsholtz, foxsauermann, hegedus, naslund}). The case of $m=p\geq 3$ being a fixed prime is of particular interest, as upper bounds for $\s(\Fpn)$ for primes $p$ imply upper bounds for $\s(\Z_m^n)$ for all integers $m$ (see \cite[Hilfssatz 2]{harborth}), and it is easy to see that $\s(\F_2^n)=2^n+1$ (see \cite[Korollar 1]{harborth}).

The problem of bounding $\s(\Fpn)$ for a fixed prime $p\geq 3$ and large $n$ is essentially equivalent to bounding the maximum size of a subset of $\Fpn$ that does not contain $p$ distinct elements summing to zero. Indeed, $\s(\Fpn)$ and the maximum size of such a set differ by a factor of at most $p$. In this article, for any fixed prime $p\geq 5$ and large $n$, we prove a new upper bound for the maximum size of a subset of $\Fpn$ without $p$ distinct elements summing to zero, stated in Theorem \ref{thm-subsets} below. Thus, we also obtain a new upper bound for $\s(\Fpn)$ for any fixed prime $p\geq 5$ and large $n$, see Corollary \ref{coro-egz} below.

For $p=3$, having three distinct elements in $\F_3^n$ summing to zero is the same as having a non-trivial three-term arithmetic progression. Determining the maximum size of a progression-free subset of $\F_3^n$ is the famous cap-set problem. In their breakthrough result from 2016, Ellenberg and Gijswijt \cite{ellengijs} proved that any subset of $\F_3^n$ without a non-trivial three-term arithmetic progression has size at most $2.756^n$. Their proof relies on a new polynomial method introduced by Croot, Lev and Pach \cite{crootlevpach} only a few days earlier. More generally, for any prime $p\geq 3$, Ellenberg and Gijswijt \cite{ellengijs} proved an upper bound of the form $\Gamma_p^n$ for the size of any subset of $\Fpn$ without a non-trivial three-term arithmetic progression. Here, $\Gamma_p<p$ is a constant only depending on $p$, which is between $0.84p$ and $0.92p$ (see \cite{blasiaketal}). In addition to the spectacular result of Ellenberg and Gijswijt \cite{ellengijs}, the polynomial method of Croot, Lev and Pach \cite{crootlevpach} has had many more applications in extremal combinatorics and additive number theory (see for example Grochow's survey \cite{grochow}).

Tao \cite{tao} introduced a reformulation of the Croot-Lev-Pach polynomial method \cite{crootlevpach}, which is now called the slice rank method. This method shows (see \cite[Theorem 4]{naslund}) that for any subset $A\su \Fpn$ of size $\vert A\vert> \gamma_p^n$ there are elements $x_1,\dots,x_p\in A$ with $x_1+\dots+x_p=0$ and such that $x_1,\dots,x_p$ are not all equal. Here,
\begin{equation}\label{eq-def-gamma}
\gamma_p=\min_{0<t<1}\frac{1+t+\dots+t^{p-1}}{t^{(p-1)/p}}
\end{equation}
is a constant just depending on $p$, and by considering $t=\frac{1}{2}$ one can see that $\gamma_p<4$ for all $p$. However, Tao's slice rank method \cite{tao} does not yield \emph{distinct} elements $x_1,\dots,x_p\in A$ with $x_1+\dots+x_p=0$. Thus, this method cannot directly be applied to the problem of bounding the size of subsets of $\Fpn$ without $p$ distinct elements summing to zero (which, for $p$ fixed and large $n$, is essentially equivalent to the Erd\H{o}s-Ginzburg-Ziv problem for $\Fpn$).

Therefore, additional ideas are required in order to apply the Croot-Lev-Pach polynomial method \cite{crootlevpach} and its consequences to the problem of bounding the size of subsets of $\Fpn$ without $p$ distinct elements summing to zero, and this has attracted the attention of several researchers \cite{foxsauermann, naslund}. On the one hand, Naslund \cite{naslund2} introduced a variation of Tao's notion of slice rank \cite{tao}, which he called partition rank, and used it \cite{naslund} to prove that any subset $A\su \Fpn$ not containing $p$ distinct elements summing to zero must have size $\vert A\vert\leq (2^p-p-2)\cdot \Gamma_p^n$. Here, $\Gamma_p$ is the above-mentioned constant only depending on $p$ which occurs in the work of Ellenberg and Gijswijt \cite{ellengijs} and lies between $0.84p$ and $0.92p$. On the other hand, relying on the result of Ellenberg and Gijswijt \cite{ellengijs} for progression-free subsets of $\F_p^{n-1}$ and a probabilistic sampling argument, Fox and the author \cite{foxsauermann} improved Naslund's bound to $\vert A\vert\leq 3\cdot \Gamma_p^n$.

Our main result in this paper is the following theorem, which (for large $n$) is a significant improvement of the previously known upper bounds. Note that while $\vert A\vert\leq p^n$ is a trivial upper bound, all previously known upper bounds were of the form $\vert A\vert\leq p^{(1-o(1))n}$, where the $o(1)$-term converges to zero as $p$ and $n$ go to infinity. In contrast, our upper bound is of the form $p^{(1/2)\cdot (1+o(1))n}$, so for large $p$ and $n$ it is roughly the square root of the previous upper bounds.

\begin{theorem}\label{thm-subsets}
Let $p\geq 5$ be a fixed prime. Then for any positive integer $n$ and any subset $A\su \Fpn$ which does not contain $p$ distinct elements $x_1,\dots,x_p\in A$ with $x_1+\dots+x_p=0$, we have
\[\vert A\vert \leq C_p\cdot \left(\sqrt{\gamma_p\cdot p}\right)^n< C_p\cdot \left(2\sqrt{p}\right)^n.\]
Here, $C_p$ is a constant only depending on $p$.
\end{theorem}

Concerning the value of the constant $C_p$, our proof gives $C_p=2p^2\cdot P(p)$, where $P(p)$ denotes the number of partitions of $p$. However, we did not optimize this constant $C_p$ in our proof.

Our proof of Theorem \ref{thm-subsets} again uses a consequence of the  Croot-Lev-Pach polynomial method \cite{crootlevpach} (or more precisely, Tao's slice rank reformulation \cite{tao}), namely the so-called multi-colored sum-free theorem. However, by combining this theorem with additional combinatorial ideas, we improve upon the bounds obtained by previous works. Indeed, for every $p\geq 5$, we have $\sqrt{\gamma_p\cdot p}<\Gamma_p$, and therefore for sufficiently large $n$ the bound in our Theorem \ref{thm-subsets} is better than the upper bounds cited above.

We remark that Theorem \ref{thm-subsets} also holds for $p=3$, but in this case it is not interesting as the result of Ellenberg and Gijswijt \cite{ellengijs} for subsets in $\F_3^n$ without a non-trivial arithmetic progression gives a better bound (recall that a non-trivial arithmetic progression in $\F_3^n$ is the same as three distinct elements summing to zero).

For any fixed prime $p\geq 5$ and large $n$, the best known lower bound for the maximum size of a subset of $\Fpn$ without $p$ distinct elements summing to zero is due to Edel \cite[Theorem 1]{edel2}. He showed, using a product construction of a particular subset of $\F_p^6$, that there is a subset of $\Fpn$ of size $\Omega(96^{n/6})$ without $p$ distinct elements summing to zero (note that $96^{n/6}\approx 2.1398^{n}$). For $p=3$ and $n$ large, there are better lower bounds, also due to Edel \cite{edel}.

As mentioned above, the maximum possible size of $\vert A\vert$ in Theorem \ref{thm-subsets} is closely related to the Erd\H{o}s-Ginzburg-Ziv constant $\s(\Fpn)$ of $\Fpn$ for a fixed prime $p\geq 5$ and large $n$. For a finite  abelian group $G$, the Erd\H{o}s-Ginzburg-Ziv constant $\s(G)$ is defined to be the smallest integer $s$ such that every sequence of $s$ (not necessarily distinct) elements of $G$ has a subsequence of length $\exp(G)$ summing to zero (here, $\exp(G)$ is the least common multiple of the orders of all elements of $G$). For $G=\Z_m^n$, the quantity $\s(\Z_m^n)$ can be interpreted geometrically as the minimum integer $s$ such that among any $s$ points in the integer lattice $\mathbb{Z}^n$ there are $m$ points whose centroid is also a lattice point in $\mathbb{Z}^n$ (see also the discussion in the beginning). Erd\H{o}s-Ginzburg-Ziv constants have been the subject of a lot of research, see for example \cite{foxsauermann} for an overview of the existing results.

Theorem \ref{thm-subsets} immediately implies a new upper bound for the Erd\H{o}s-Ginzburg-Ziv constant $\s(\Fpn)$ of $\Fpn$ for a prime $p\geq 5$ and large $n$ in the following way. Every sequence of vectors in $\Fpn$ without a zero-sum subsequence of length $p$ can contain any vector in $\Fpn$ at most $p-1$ times. Furthermore, the set $A\su \Fpn$ of all vectors occurring at least once in the sequence satisfies the assumptions of Theorem \ref{thm-subsets}. Thus, Theorem \ref{thm-subsets} yields the following corollary.

\begin{corollary}\label{coro-egz}
For a fixed prime $p\geq 5$ and any positive integer $n$, we have
\[\s(\Fpn)\leq (p-1)\cdot C_p\cdot \left(\sqrt{\gamma_p\cdot p}\right)^n+1< (p-1)\cdot C_p\cdot \left(2\sqrt{p}\right)^n+1.\]
\end{corollary}

For sufficiently large $n$, this corollary improves the best previous upper bounds for $\s(\Fpn)$. Using \cite[Lemma 11]{foxsauermann}, Corollary \ref{coro-egz} also implies improved upper bounds for $\s(\Z_m^n)$ for a fixed integer $m$ with a prime factor $p\geq 5$ and large $n$, as well as for  $\s(G)$ for many other abelian groups $G$.

We remark that our proof also gives a multi-colored version of Theorem \ref{thm-subsets}, and in this version the upper bound is close to tight when $n$ is large. For more details, see Section \ref{sect-concluding-remarks}.

\textit{Organization.} In Section \ref{sect-overview}, we will first state the so-called multi-colored sum-free theorem which will be used to prove Theorem \ref{thm-subsets}. We will then give a rough overview of our proof of Theorem \ref{thm-subsets}. The actual proof is contained in Sections \ref{sect-proof-propo-egz} and \ref{sect-proof-thm-subset}. Finally, Section \ref{sect-concluding-remarks} contains some concluding remarks.

\section{Proof Overview}
\label{sect-overview} 

Our proof of Theorem \ref{thm-subsets} uses the multi-colored sum-free theorem, see Theorem \ref{thm-multicolor-sumfree} below. In the case of only three variables, it was observed by Blasiak, Church, Cohn, Grochow, Naslund, Sawin and Umans \cite{blasiaketal} that the argument of Ellenberg and Gijswijt \cite{ellengijs} for progression-free subsets of $\Fpn$ also carries over to the more general situation of 3-colored sum-free sets in $\Fpn$. In the general case, the multi-colored sum-free theorem is a direct consequence of Tao's slice rank method \cite{tao} (which generalizes the Croot-Lev-Pach polynomial method \cite{crootlevpach} to more than three variables).

The following definition was introduced by Alon, Shpilka and Umans \cite{alon-shpilka-umans} for the case $k=3$.

\begin{definition} Let $G$ be an abelian group and let $k\geq 3$. A \emph{$k$-colored sum-free set} in $G$ is a collection of $k$-tuples $(x_{1,i}, x_{2,i}, \dots, x_{k,i})_{i=1}^L$ of elements of $G$ such that for all $i_1,\dots, i_k\in \lbrace 1,\dots, L\rbrace$
\[x_{1,i_1}+x_{2,i_2}+\dots +x_{k,i_k}=0 \quad \text{ if and only if }\quad  i_1=i_2=...=i_k.\]
The size of a $k$-colored sum-free set is the number of $k$-tuples it consists of.
\end{definition}

If $G=\Fpn$, then Tao's slice rank formulation \cite{tao} of the Croot-Lev-Pach polynomial method \cite{crootlevpach} yields an upper bound for the size of a $k$-colored sum-free set in $\Fpn$ for any $k\geq 3$ (and if $p$ and $k$ are fixed and $n$ is large, then this bound is essentially tight \cite{lovaszsauermann}). To state this bound, let us define
\[ \Gamma_{p,k}= \min_{0<t<1}\frac{1+t+\dots+t^{p-1}}{t^{(p-1)/k}}.\]
It is not hard to see that this minimum exists and that $\Gamma_{p,k}<p$. Note that for $k=p$, we have $\Gamma_{p,p}=\gamma_p$ (recall that $\gamma_p$  was defined in (\ref{eq-def-gamma})). Furthermore, the constant $\Gamma_p$ in the work of Ellenberg and Gijswijt \cite{ellengijs} mentioned in the introduction equals $\Gamma_{p,3}$.

\begin{theorem}\label{thm-multicolor-sumfree}
Fix $k\geq 3$ and a prime $p\geq 3$. Then, for any positive integer $n$, the size of any $k$-colored sum-free set in $\Fpn$ is at most $\Gamma_{p,k}^n$.
\end{theorem}

As mentioned above, Theorem \ref{thm-multicolor-sumfree} is a straightforward application of Tao's slice rank method \cite{tao}. The details of the proof can be found in \cite[Section 9]{lovaszsauermann}.

We will now give an overview of the proof of Theorem \ref{thm-subsets}.

In order to prove Theorem \ref{thm-subsets}, let us fix a prime $p\geq 5$ and a positive integer $n$. Let us call a $p$-tuple $(x_1,\dots,x_p)\in \Fpn\times\dots\times \Fpn$ a \emph{cycle} if $x_1+\dots+x_p=0$. Furthermore, let us call two cycles $(x_1,\dots,x_p), (x_1',\dots,x_p')\in \Fpn\times\dots\times \Fpn$ disjoint if the sets $\lbrace x_1,\dots,x_p\rbrace$ and $\lbrace x_1',\dots,x_p'\rbrace$ are disjoint. In other words, the cycles $(x_1,\dots,x_p)$ and $(x_1',\dots,x_p')$ are disjoint if no element of $\Fpn$ appears in both of them (but note that each of the cycles is allowed to contain an element of $\Fpn$  multiple times).

We will first prove the following proposition.

\begin{proposition}\label{propo-egz}
Suppose that $X_1,\dots,X_p$ are subsets of $\Fpn$ such that every cycle $(x_1,\dots,x_p)\in X_1\times\dots\times X_p$ satisfies $x_1=x_2$. Furthermore, suppose that for some positive integer $L$, there exists a collection of $L$ disjoint cycles in $X_1\times\dots\times X_p$. Then we must have
\[L\leq 2p\cdot \left(\sqrt{\gamma_p\cdot p}\right)^n.\]
\end{proposition}

The proof of Proposition \ref{propo-egz} can be found in Section \ref{sect-proof-propo-egz}. The idea behind the proof is to show that for any $j=3,\dots,p$, the number of pairs $(x_1,x_j)\in X_1\times X_j$ that appear in some cycle $(x_1,\dots,x_p)\in X_1\times\dots\times X_p$ is not too large. This will enable us to find a large $p$-colored sum-free set within a collection of $L$ disjoint cycles in $X_1\times\dots\times X_p$. Then, applying Theorem \ref{thm-multicolor-sumfree} with $k=p$ will give the desired bound on $L$.

In Section \ref{sect-proof-thm-subset} we will deduce Theorem \ref{thm-subsets} from Proposition \ref{propo-egz}. In order to do so, we will take a subset $A\su \Fpn$ as in Theorem \ref{thm-subsets} and consider all cycles $(x_1,\dots,x_p)\in A\times\dots\times A$. By assumption, every such cycle contains some element of $\Fpn$ at least twice. For a given cycle in $A\times\dots\times A$, we obtain a pattern of how many different elements of $\Fpn$ occur in this cycle and with which multiplicities the different elements occur. We will then go through all the different possibilities of such patterns (in a suitably chosen order). For each pattern, we will either find a large collection of disjoint cycles with that pattern, or we will be able to delete few elements from $A$ and destroy all cycles with that pattern. This will enable us to apply Proposition \ref{propo-egz}. 

\section{Proof of Proposition \ref{propo-egz}}
\label{sect-proof-propo-egz}

Recall that we fixed a prime $p\geq 5$ and a positive integer $n$. In the last section, we defined the notion of cycles and what it means for two cycles to be disjoint.

Let us fix subsets $X_1,\dots,X_p\su \Fpn$ as in the statement of Proposition \ref{propo-egz}. For $j=3,\dots,p$, let us say that a pair $(y,z)\in X_1\times X_j$ is $j$-extendable if there exists a cycle $(x_1,\dots,x_p)\in X_1\times\dots\times X_p$ with $x_1=y$ and $x_j=z$.

\begin{lemma}\label{lem-j-extendable}
For every $j=3,\dots,p$, the number of $j$-extendable pairs $(y,z)\in X_1\times X_j$ is at most $p^n$.
\end{lemma}
\begin{proof}
Let us assume without loss of generality that $j=3$. We claim that for any two distinct $3$-extendable pairs $(y,z), (y',z')\in X_1\times X_3$ we have $y+z\neq y'+z'$. This immediately implies that the total number of $3$-extendable pairs $(y,z)\in X_1\times X_3$ is at most $\vert \Fpn\vert=p^n$.

So suppose for contradiction that there exist two distinct $3$-extendable pairs $(y,z), (y',z')\in X_1\times X_3$ with $y+z= y'+z'$. Then we must have $y\neq y'$ since otherwise by $y+z= y'+z'$ we would also have $z=z'$ and therefore $(y,z)=(y',z')$.

Since the pair $(y,z)$ is $3$-extendable, there exists a cycle $(x_1,\dots,x_p)\in X_1\times\dots\times X_p$ with $x_1=y$ and $x_3=z$. Note that by the assumption on the sets $X_1,\dots,X_p\su \Fpn$ in Proposition \ref{propo-egz}, we have $x_2=x_1=y$. But note that $(y',x_2, z', x_4,\dots,x_p)\in X_1\times\dots\times X_p$. Furthermore, $y+z= y'+z'$ implies
\[y'+x_2+z'+x_4+\dots+x_p=y+x_2+z+x_4+\dots+x_p=x_1+x_2+x_3+x_4+\dots+x_p=0,\]
so $(y',x_2, z', x_4,\dots,x_p)\in X_1\times\dots\times X_p$ is a cycle. Thus, again using the assumption on the sets $X_1,\dots,X_p\su \Fpn$ in Proposition \ref{propo-egz}, we must have $x_2=y'$. This contradicts $x_2=y$ and $y\neq y'$.
\end{proof}

Recall that we assumed in Proposition \ref{propo-egz} that there exists a collection of $L$ disjoint cycles in $X_1\times\dots\times X_p$. So let $\M\su X_1\times\dots\times X_p$ be such a collection, then $\vert \M\vert=L$. Note that we clearly have $L\leq p^n$ since all the cycles in $\M$ are disjoint.

\begin{lemma}\label{lemma-subcollection}
There exists a subcollection $\M'\su \M\su X_1\times\dots\times X_p$ of size $\vert \M'\vert\geq L^2/(2p^{n+1})$ such that for any two distinct cycles $(x_1,\dots,x_p), (x_1',\dots,x_p')\in \M'$ and any $j\in \lbrace 3,\dots,p\rbrace$, the pair $(x_1,x_j')$ is not $j$-extendable.
\end{lemma}
\begin{proof}Let us define a graph $G$ with vertex set $\M$ in the following way: For any two distinct cycles $(x_1,\dots,x_p), (x_1',\dots,x_p')\in \M$, we draw an edge edge between the vertices $(x_1,\dots,x_p)$ and $(x_1',\dots,x_p')$ if for some $j\in \lbrace 3,\dots,p\rbrace$ the pair $(x_1,x_j')$ or the pair $(x_1',x_j)$ is $j$-extendable. Then we need to show that there is a subcollection $\M'\su \M$ of size $\vert \M'\vert\geq L^2/(2p^{n+1})$ such that there are no edges between any two vertices in $\M'$ (i.e.\ such that $\M'$ is an independent set in the graph $G$).

Recall that by Lemma \ref{lem-j-extendable} for each $j\in \lbrace 3,\dots,p\rbrace$ the total number of $j$-extendable pairs is at most $p^n$. Since $\M$ is a collection of disjoint cycles, each $j$-extendable pair causes at most one edge in the graph $G$. Thus, the total number of edges in $G$ is at most $(p-2)\cdot p^n$.

As $G$ has $\vert \M\vert=L$ vertices, this means that the average degree $d$ of the graph $G$ satisfies $d\leq 2(p-2)p^n/L$ and consequently (recalling that $L\leq p^n$)
\[d+1\leq \frac{2(p-2)p^n}{L}+1=\frac{2(p-2)p^n+L}{L}\leq \frac{2p\cdot p^n}{L}.\]

Now, $G$ is a graph with $\vert \M\vert=L$ vertices and average degree $d$. Denoting the degree of each vertex $v\in V(G)=\M$ by $d(v)$, the graph $G$  contains an independent set of size at least
\[\sum_{v\in V(G)}\frac{1}{d(v)+1}\geq L\cdot \frac{1}{d+1}\geq \frac{L^2}{2p^{n+1}}.\]
Indeed, the existence of an independent set of size at least $\sum_{v\in V(G)}\frac{1}{d(v)+1}$ is a well-known result due to Caro \cite{caro} and Wei \cite{wei} (see also, for example, \cite[p.\ 100]{alon-spencer}), and the first inequality above follows from Jensen's inequality. Thus, we can indeed find an independent set $\M'\su \M$ of size at least $\vert \M'\vert\geq L^2/(2p^{n+1})$.
\end{proof}

Let $\M'\su \M\su X_1\times\dots\times X_p$ be a subcollection as in Lemma \ref{lemma-subcollection}, and let $(x_{1,i}, x_{2,i}, \dots, x_{p,i})$ for $i=1,\dots,\vert \M'\vert$ be the cycles in $\M'$. Note that these cycles are all disjoint, since $\M'\su \M$ and $\M$ is a collection of disjoint cycles.

\begin{lemma}\label{lemma-M-strich}
The collection $\M'=(x_{1,i}, x_{2,i}, \dots, x_{p,i})_{i=1}^{\vert \M'\vert}$ is a $p$-colored sum-free set in $\Fpn$. 
\end{lemma}
\begin{proof}
For every $i=1,\dots, \vert \M'\vert$, the $p$-tuple $(x_{1,i}, x_{2,i}, \dots, x_{p,i})$ is a cycle and therefore we have $x_{1,i}+ x_{2,i}+ \dots+ x_{p,i}=0$. It remains to show that for all indices $i_1,\dots i_p\in \lbrace 1,\dots, \vert \M'\vert\rbrace$ with $x_{1,i_1}+ x_{2,i_2}+ \dots+ x_{p,i_p}=0$ we have $i_1=\dots=i_p$.

Suppose the contrary, then there exist $i_1,\dots i_p\in \lbrace 1,\dots, \vert \M'\vert\rbrace$ with $x_{1,i_1}+ x_{2,i_2}+ \dots+ x_{p,i_p}=0$ and such that $i_1,\dots i_p$ are not all equal. For every $j=1,\dots,p$, we have $(x_{1,i_j}, x_{2,i_j}, \dots, x_{p,i_j})\in \M'\su X_1\times\dots\times X_p$ and therefore in particular $x_{j,i_j}\in X_j$. Thus, $(x_{1,i_1}, x_{2,i_2}, \dots, x_{p,i_p})\in X_1\times\dots\times X_p$. Furthermore, by the assumption $x_{1,i_1}+ x_{2,i_2}+ \dots+ x_{p,i_p}=0$, the $p$-tuple $(x_{1,i_1}, x_{2,i_2}, \dots, x_{p,i_p})$ is a cycle. Thus, by the assumption of Proposition \ref{propo-egz}, we must have $x_{1,i_1}=x_{2,i_2}$ However, since the different cycles in $\M'$ are all disjoint, this is only possible if $i_1=i_2$.

Therefore, as we assumed that $i_1,\dots i_p$ are not all equal, there must be some $j\in \lbrace 3,\dots,p\rbrace$ with $i_j\neq i_1$. Then $(x_{1,i_1}, x_{2,i_1}, \dots, x_{p,i_1})$ and $(x_{1,i_j}, x_{2,i_j}, \dots, x_{p,i_j})$ are distinct cycles in $\M'$. Hence, by the condition on $\M'$ in Lemma \ref{lemma-subcollection}, the pair $(x_{1,i_1}, x_{j,i_j})$ is not $j$-extendable. On the other hand, the cycle $(x_{1,i_1}, x_{2,i_2}, \dots, x_{p,i_p})\in X_1\times\dots\times X_p$ establishes that the pair $(x_{1,i_1}, x_{j,i_j})$ is $j$-extendable. This is a contradiction.
\end{proof}

Combining Lemma \ref{lemma-M-strich} with Theorem \ref{thm-multicolor-sumfree} for $k=p$, we obtain that $\vert \M'\vert\leq \Gamma_{p,p}^n=\gamma_p^n$. Together with the condition $\vert \M'\vert\geq L^2/(2p^{n+1})$ in Lemma \ref{lemma-subcollection}, this gives
\[\frac{L^2}{2p^{n+1}}\leq \vert \M'\vert\leq \gamma_p^n.\]
Rearranging yields $L^2\leq 2p\cdot (\gamma_p\cdot p)^n$ and therefore $L\leq \sqrt{2p}\cdot \left(\sqrt{\gamma_p\cdot p}\right)^n\leq 2p\cdot \left(\sqrt{\gamma_p\cdot p}\right)^n$, as desired. This finishes the proof of Proposition \ref{propo-egz}.

\section{Proof of Theorem \ref{thm-subsets}}
\label{sect-proof-thm-subset}

In this section we deduce Theorem \ref{thm-subsets} from Proposition \ref{propo-egz}. Recall that we already fixed  a prime $p\geq 5$ and a positive integer $n$. Furthermore, let us fix a subset $A\su \Fpn$ not containing $p$ distinct elements $x_1,\dots,x_p\in A$ with $x_1+\dots+x_p=0$. This means that every cycle $(x_1,\dots,x_p)\in A\times \dots\times A$ contains some element of $\Fpn$ at least twice.

Now, for every cycle $(x_1,\dots,x_p)$ in $\Fpn$ let us consider its \emph{multiplicity pattern} $\lambda=(\lambda_1,\dots, \lambda_k)$ which is given as follows: Let $k$ be the number of distinct elements of $\Fpn$ among $x_1,\dots,x_p$ and let $\lambda_1\geq \lambda_2\geq \dots\geq \lambda_k>0$ be the multiplicities with which these elements occur among $x_1,\dots,x_p$. Furthermore, we call the number $k$ of distinct elements of $\Fpn$ among $x_1,\dots,x_p$ the \emph{length} of the multiplicity pattern $\lambda=(\lambda_1,\dots, \lambda_k)$ of the cycle $(x_1,\dots,x_p)$.

For example, if $p=7$, and the cycle $(x_1,\dots,x_7)$ contains one elements of $\F_7^n$ three times, another element twice and the two remaining elements are distinct, then its multiplicity pattern is $(3,2,1,1)$ and has length 4.

Note that the multiplicity pattern $\lambda=(\lambda_1,\dots, \lambda_k)$ of any cycle $(x_1,\dots,x_p)$ satisfies $\lambda_1+\dots +\lambda_k=p$ and is therefore a partition of $p$. Thus, the total number of possible multiplicity patterns equals  the number $P(p)$ of partitions of $p$. Let
\[L=\left\lceil\frac{\vert A\vert}{p\cdot P(p)}\right\rceil,\]
and note that we may assume that $L>0$ (otherwise $\vert A\vert =0$, in which case Theorem \ref{thm-subsets} is clearly true).

The following lemma states, roughly speaking, that we can find a subset $A'\su A$ such that within $A'\times \dots\times A'$ there is a large collection of disjoint cycles with the same multiplicity pattern, but no cycle with a multiplicity pattern of a bigger length. We can then use this large collection of disjoint cycles with the same multiplicity pattern to construct sets $X_1,\dots,X_p$ as in Proposition \ref{propo-egz}.

\begin{lemma}\label{lemma-subset-pattern}
There exists a subset $A'\su A$ and a multiplicity pattern $\lambda=(\lambda_1,\dots, \lambda_k)$ such that the following two conditions hold:
\begin{itemize}
\item There is a collection of $L$ disjoint cycles $(x_1,\dots,x_p)\in A'\times \dots\times A'$ with multiplicity pattern $\lambda$.
\item For every cycle $(x_1,\dots,x_p)\in A'\times \dots\times A'$, the length of the multiplicity pattern of $(x_1,\dots,x_p)$ is not larger than the length $k$ of the multiplicity pattern $\lambda$.
\end{itemize}
\end{lemma}
\begin{proof}
Let us make a list $\lambda^{(1)}, \lambda^{(2)},\dots, \lambda^{(P(p))}$ of all the $P(p)$ possible multiplicity patterns, ordered by decreasing length (the order of multiplicity patterns of the same length is chosen arbitrarily).

Start with the set $A'=A$. Now, let us go through the multiplicity patterns one by one. In step $i$, when considering the multiplicity pattern $\lambda^{(i)}$, there are two options: If there exists a collection of $L$ disjoint cycles $(x_1,\dots,x_p)\in A'\times \dots\times A'$ with multiplicity pattern $\lambda^{(i)}$, stop the procedure. Otherwise, choose a maximal collection of disjoint cycles $(x_1,\dots,x_p)\in A'\times \dots\times A'$ with multiplicity pattern $\lambda^{(i)}$ and delete all elements that occur in a cycle in this collection from the set $A'$ (and then go to the next step).

Note that whenever the second option occurs for some $\lambda^{(i)}$, the maximal collection of disjoint cycles $(x_1,\dots,x_p)\in A'\times \dots\times A'$ has size at most $L-1$, and therefore at most $p(L-1)$ elements get deleted from the set $A'$. We claim that after this deletion, there are no cycles $(x_1,\dots,x_p)\in A'\times \dots\times A'$ with multiplicity pattern $\lambda^{(i)}$. Indeed if there was any such cycle $(x_1,\dots,x_p)$ left after the deletion, then this cycle could have been added to the maximal collection of disjoint cycles $(x_1,\dots,x_p)\in A'\times \dots\times A'$ that we chose during step $i$. This would be a contradiction to the maximality of the chosen collection. Thus, if the second option occurs in step $i$, then after step $i$ there are no cycles $(x_1,\dots,x_p)\in A'\times \dots\times A'$ with multiplicity pattern $\lambda^{(i)}$.

Suppose that the entire procedure runs through all $P(p)$ steps without the first option ever occurring (which would cause us to stop the procedure). Then after all $P(p)$ steps are completed, the remaining set $A'$ does not have any cycle $(x_1,\dots,x_p)\in A'\times \dots\times A'$ with multiplicity pattern $\lambda^{(i)}$ for any $i=1,\dots,P(p)$. This means that there cannot exist any cycles $(x_1,\dots,x_p)\in A'\times \dots\times A'$ at all. On the other hand, in each of the $P(p)$ steps we deleted at most $p(L-1)$ elements from the set $A'$ and we started with $A'=A$. Thus, after the $P(p)$ steps, the remaining set $A'$ has size
\[\vert A'\vert\geq \vert A\vert-P(p)\cdot p(L-1)>\vert A\vert-P(p)\cdot p\cdot\frac{\vert A\vert}{p\cdot P(p)}=0.\]
Thus, the remaining set $A'$ is non-empty. But taking any element $x\in A'$ we find a cycle $(x,\dots,x)\in A'\times\dots\times A'$ (as $x+\dots+x=p\cdot x=0$ in $\Fpn$). This is a contradiction.

Thus, for some $i=1,\dots, P(p)$, the first option occurs in step $i$. We claim that then $\lambda=\lambda^{(i)}$ and the set $A'\su A$ at step $i$ satisfy the conditions in the lemma. Since the first option occurs in step $i$, there exists a collection of $L$ disjoint cycles $(x_1,\dots,x_p)\in A'\times \dots\times A'$ with multiplicity pattern $\lambda^{(i)}$. For the second condition, recall that for all $i'<i$ after step $i'$ there are no no cycles $(x_1,\dots,x_p)\in A'\times \dots\times A'$ with multiplicity pattern $\lambda^{(i')}$. Thus, our final set $A'$ at step $i$ does not have cycles $(x_1,\dots,x_p)\in A'\times \dots\times A'$ with multiplicity pattern $\lambda^{(i')}$ for any $i'<i$. In particular, there are no cycles $(x_1,\dots,x_p)\in A'\times \dots\times A'$ such that the multiplicity pattern of $(x_1,\dots,x_p)$ has length larger than the length of $\lambda^{(i)}$ (because then, by the ordering of $\lambda^{(1)}, \lambda^{(2)},\dots, \lambda^{(P(p))}$ by decreasing length, the multiplicity pattern of $(x_1,\dots,x_p)$ would equal $\lambda^{(i')}$ for some $i'<i$). Thus, the second condition is satisfied as well. This finishes the proof of the lemma.
\end{proof}

Let us fix $A'\su A$ and $\lambda=(\lambda_1,\dots, \lambda_k)$ as in Lemma \ref{lemma-subset-pattern}. Furthermore, take a collection $\M$ of $L$ disjoint cycles $(x_1,\dots,x_p)\in A'\times \dots\times A'$ with multiplicity pattern $\lambda$. Note that for any cycle $(x_1,\dots,x_p)$, any reordering of $(x_1,\dots,x_p)$ is also a cycle. Thus, we can assume that each cycle $(x_1,\dots,x_p)\in \M$ is ordered in such a way that the first $\lambda_1$ vectors are equal, the next $\lambda_2$ vectors are equal and so on.

Recall that by the assumption on $A$ in Theorem \ref{thm-subsets}, every cycle $(x_1,\dots,x_p)\in A\times \dots\times A$ contains some vector at least twice. In particular, every $(x_1,\dots,x_p)\in \M$ contains some vector at least twice. Thus, the multiplicity pattern $\lambda=(\lambda_1,\dots, \lambda_k)$ has some entry $\lambda_i\geq 2$. Hence $\lambda_1\geq 2$ as $\lambda_1\geq \lambda_2\geq \dots\geq \lambda_k$.

Let us divide the index set $\lbrace 1,\dots, p\rbrace$ into $k$ \emph{blocks} as follows: The first $\lambda_1$ indices (namely the indices $1, \dots, \lambda_1$) form the first block, the next $\lambda_2$ indices the second block and so on. Then, for any cycle $(x_1,\dots,x_p)\in \M$ we have $x_i=x_j$ if and only if $i$ and $j$ are in the same block.

Now, for $i=1,\dots,p$, define
\[X_i=\lbrace x_i\mid (x_1,\dots,x_p)\in \M\rbrace.\]
As $(x_1,\dots,x_p)\in A'\times \dots\times A'$ for all $(x_1,\dots,x_p)\in \M$, we have $X_i\su A'$ for all $i=1,\dots,p$.

We claim that for any $i,j\in \lbrace 1,\dots,p\rbrace$ such that $i$ and $j$ are in different blocks, the sets $X_i$ and $X_j$ are disjoint. Indeed, assume that an element $x\in \Fpn$ occurs as $x_i$ in some cycle $(x_1,\dots,x_p)\in \M$ and as $x_j'$ in some cycle $(x_1',\dots,x_p')\in \M$. Since any two cycles in $\M$ are disjoint, we must have $(x_1,\dots,x_p)=(x_1',\dots,x_p')$ and therefore $x_j=x=x_i$. But this contradicts $i$ and $j$ being in different blocks. Thus, the sets $X_i$ and $X_j$ are indeed disjoint for any $i,j\in \lbrace 1,\dots,p\rbrace$ such that $i$ and $j$ are in different blocks.

\begin{lemma}\label{lem-cycle-x1-x2}
Every cycle $(x_1,\dots,x_p)\in X_1\times \dots\times X_p$ satisfies $x_1=x_2$.
\end{lemma} 
\begin{proof}
Let $(x_1,\dots,x_p)\in X_1\times \dots\times X_p$ be a cycle. Since $X_i\su A'$ for $i=1,\dots,p$, we have $(x_1,\dots,x_p)\in A'\times \dots\times A'$ and therefore, by the second condition in Lemma \ref{lemma-subset-pattern}, the multiplicity pattern of $(x_1,\dots,x_p)$ has length at most $k$. In other words, there are at most $k$ distinct elements appearing among $x_1,\dots,x_p$.

On the other hand, recall that for any $i,j\in \lbrace 1,\dots,p\rbrace$ such that $i$ and $j$ are in different blocks, the sets $X_i$ and $X_j$ are disjoint. Thus, for all such $i$ and $j$ we must have $x_i\neq x_j$. Hence, since there are $k$ blocks, there must be at least $k$ distinct elements appearing among $x_1,\dots,x_p$. So we can conclude that there are exactly $k$ distinct elements appearing among $x_1,\dots,x_p$. Furthermore, whenever $i$ and $j$ are in the same block, we must have $x_i=x_j$ (since otherwise we would obtain at least $k+1$ distinct elements appearing among $x_1,\dots,x_p$).

Because of $\lambda_1\geq 2$, the indices 1 and 2 are in the same block. Hence $x_1=x_2$, as desired.\end{proof}
 
By Lemma \ref{lem-cycle-x1-x2} the sets $X_1,\dots, X_p\su A'\su \Fpn$ satisfy the assumption of Proposition \ref{propo-egz}. Note that by the definition of the sets $X_i$ we clearly have $(x_1,\dots,x_p)\in X_1\times \dots\times X_p$ for every $(x_1,\dots,x_p)\in \M$. Therefore $\M$ is a collection of $L$ disjoint cycles in $X_1\times \dots\times X_p$. Thus,  Proposition \ref{propo-egz} implies that
\[\frac{\vert A\vert}{p\cdot P(p)}\leq L\leq 2p\cdot \left(\sqrt{\gamma_p\cdot p}\right)^n.\]
Consequently, $\vert A\vert\leq 2p^2\cdot P(p)\cdot \left(\sqrt{\gamma_p\cdot p}\right)^n$. This proves the first inequality in Theorem \ref{thm-subsets} with $C_p=2p^2\cdot P(p)$.

The second inequality in Theorem \ref{thm-subsets} follows from the easy observation that
\[\gamma_p=\min_{0<t<1}\frac{1+t+\dots+t^{p-1}}{t^{(p-1)/p}}\leq  \frac{1+(1/2)+\dots+(1/2)^{p-1}}{(1/2)^{(p-1)/p}}< \frac{2}{1/2}=4.\]

\section{Concluding Remarks}
\label{sect-concluding-remarks}

It is not clear whether the bound in Theorem \ref{thm-subsets} is anywhere close to optimal. As mentioned in the introduction, when forbidding solutions to $x_1+\dots+x_p=0$ whenever $x_1,\dots,x_p\in A$ are not all equal, Tao's slice rank method \cite{tao} gives the much stronger bound $\vert A\vert\leq \gamma_p^n<4^n$. It might be the case that a bound of the form $c^n$ for some absolute constant $c$ also holds in the setting of Theorem \ref{thm-subsets}. The best known lower bounds, due to Edel \cite{edel2}, are of this form with $c$ being roughly $2.1398$.

One may also consider the following multi-colored generalization of Theorem \ref{thm-subsets}.

\begin{theorem}\label{thm-subsets-multicolored}
Let us fix an integer $k\geq 3$ and a prime $p\geq 5$. For any positive integer $n$, consider a collection of $k$-tuples $(x_{1,i}, x_{2,i}, \dots, x_{k,i})_{i=1}^L$ of elements of $\Fpn$ such that for each $j=1,\dots,k$ all the elements $x_{j,i}$ for $i\in \lbrace 1,\dots,L\rbrace$ are distinct. Furthermore assume that
\[x_{1,i}+x_{2,i}+\dots +x_{k,i}=0\]
for every $i=1,\dots,L$ and that there are no distinct indices $i_1,\dots, i_k\in \lbrace 1,\dots, L\rbrace$ with
\[x_{1,i_1}+x_{2,i_2}+\dots +x_{k,i_k}=0.\]
Then we have
\[L \leq C_{k}'\cdot \left(\sqrt{\Gamma_{p,k}\cdot p}\right)^n,\]
where $C_{k}'$ is a constant only depending on $k$.
\end{theorem}

This multicolored version is a generalization of Theorem \ref{thm-subsets}. Indeed, we can recover Theorem \ref{thm-subsets} by taking $k=p$, $L=\vert A\vert$ and taking the $p$-tuples $(x_{1,i},x_{2,i}, \dots, x_{k,i})$ to be $(x,\dots,x)$ for all $x\in A$. One can prove Theorem \ref{thm-subsets-multicolored} in essentially the same way as Theorem \ref{thm-subsets}. The only difference is that instead of multiplicity patterns (which were partitions of $p$), one needs to consider partitions of the set $\lbrace 1,\dots, k\rbrace$ (where for a given cycle $(x_{1,i_1},x_{2,i_2},\dots,x_{k,i_k})$ two indices $j,j'\in \lbrace 1,\dots, p\rbrace$ lie in the same partition class if and only if $i_j=i_{j'}$). Then one can reduce Theorem \ref{thm-subsets-multicolored} to a multi-colored version of Proposition \ref{propo-egz} with the same arguments as in Section \ref{sect-proof-thm-subset}. This multi-colored version of Proposition \ref{propo-egz} can be proved with the same arguments as in Section \ref{sect-proof-propo-egz}. All in all, one obtains Theorem \ref{thm-subsets-multicolored} with $C_{k}'=2k^2\cdot B(k)$, where $B(k)$ denotes the number of partitions of the set $\lbrace 1,\dots, k\rbrace$. However, the arguments in this multi-colored setting are notationally much more confusing and the statement of Theorem \ref{thm-subsets-multicolored} is less natural than the one of Theorem \ref{thm-subsets}. Therefore, we decided to present the proof in the setting of Theorem \ref{thm-subsets}.

In Theorem \ref{thm-subsets-multicolored} for $k=p$, we obtain $L \leq C_{k}'\cdot \left(\sqrt{\Gamma_{p,k}\cdot p}\right)^n \leq C_{k}'\cdot \left(\sqrt{\gamma_{p}\cdot p}\right)^n \leq C_{k}'\cdot \left(2\sqrt{p}\right)^n$, which is close to optimal. Indeed, if $n$ is even, it is possible to find a collection of $p$-tuples $(x_{1,i}, x_{2,i}, \dots, x_{p,i})_{i=1}^L$ with $L=p^{n/2}=\sqrt{p}^n$ satisfying the conditions in Theorem \ref{thm-subsets-multicolored}. Indeed, let $y_1,\dots,y_L$ be a list of all the vectors in $\Fpn$ whose first $n/2$ coordinates are zero, and let $z_1,\dots,z_L$ be a list of all the vectors in $\Fpn$ whose last $n/2$ coordinates are zero. Now for $i=1,\dots,L$, take $(x_{1,i}, x_{2,i}, \dots, x_{p,i})=(y_{i}, -y_{i}, z_{i}, \dots, z_{i}, -(p-3)z_i)$. It is not hard to see that this collection satisfies all the conditions (for the last condition, note that whenever $x_{1,i_1}+x_{2,i_2}+\dots +x_{p,i_p}=0$ for some indices $i_1,\dots, i_p\in \lbrace 1,\dots, L\rbrace$, we must have $i_1=i_2$). Thus, it is indeed possible to find a collection of $p$-tuples $(x_{1,i}, x_{2,i}, \dots, x_{p,i})_{i=1}^L$ as in Theorem \ref{thm-subsets-multicolored} for $k=p$ with $L=p^{n/2}=\sqrt{p}^n$.

In particular, in Theorem \ref{thm-subsets-multicolored} for $k=p$ one cannot have a bound of the form $c^n$ for some absolute constant $c$. Thus, if one hopes to prove such a bound for Theorem \ref{thm-subsets}, which is a special case of Theorem \ref{thm-subsets-multicolored} for $k=p$, the proof would need to take a different route.

\textit{Acknowledgements.} The author would like thank her advisor Jacob Fox for several helpful discussions and many useful comments on previous versions of this paper. Furthermore, the author would like to thank Masato Mimura, Yuta Suzuki and Norihide Tokushige for pointing out an error in the proof of Lemma \ref{lemma-subcollection} in a previous version of this paper. The author is also grateful to Christian Elsholtz for drawing her attention to several of the references. Finally, the author would like to thank the anonymous referee for their careful reading of the paper and their helpful comments.

\end{document}